\documentclass[12pt, reqno]{amsart}
\usepackage{amsmath, amsthm, amscd, amsfonts, amssymb, graphicx, color}
\usepackage[bookmarksnumbered, colorlinks, plainpages]{hyperref}
\hypersetup{colorlinks=true,linkcolor=red, anchorcolor=green, citecolor=cyan, urlcolor=red, filecolor=magenta, pdftoolbar=true}

\newtheorem{theorem}{Theorem}[section]

\theoremstyle{definition}

\newtheorem{example}[theorem]{Example}

\theoremstyle{remark}

\numberwithin{equation}{section}

\begin{document}
\setcounter{page}{1}

\title[Polytopes of Stochastic Tensors]{Polytopes of Stochastic Tensors}

\author[H. Chang, V.E. Paksoy, F. Zhang]{Haixia Chang$^1$, Vehbi E. Paksoy$^2$ and Fuzhen Zhang$^2$$^{*}$}

\address{$^{1}$ School of Statistics and Mathematics, Shanghai Finance University, Shanghai 201209, P.R. China.}
\email{\textcolor[rgb]{0.00,0.00,0.84}{hcychang@163.com}}

\address{$^{2}$ Department of Mathematics, Nova Southeastern University,
3301 College Ave., Fort Lauderdale, FL 33314, USA.}
\email{\textcolor[rgb]{0.00,0.00,0.84}{vp80@nova.edu; zhang@nova.edu}}




\date{Received: xxxxxx; Revised: yyyyyy; Accepted: zzzzzz.
\newline \indent $^{*}$ Corresponding author}

\begin{abstract}
Considering $n\times n\times n$ stochastic tensors $(a_{ijk})$
(i.e., nonnegative hypermatrices in which every sum over one index $i$, $j$, or $k$,  is 1),
we study the polytope ($\Omega_{n}$) of all these tensors, the convex set ($L_n$) of
all tensors in $\Omega_{n}$ with some positive diagonals, and the polytope ($\Delta_n$)
generated by the permutation tensors. We show that $L_n$ is almost the same as
$\Omega_{n}$ except for some boundary points. We also present an upper bound
for the number of vertices of $\Omega_{n}$.
\end{abstract} \maketitle

\section{Introduction}

A square matrix is doubly stochastic if its entries are all nonnegative
and each row and column sum is 1.
A celebrated result known as  Birkhoff's theorem about doubly stochastic matrices
(see, e.g., \cite[p.~549]{HJ1.13}) states that
an $n\times n$ matrix is doubly stochastic if and only if it is a convex combination of some $n\times n$ permutation matrices.
Considered as elements in $\Bbb R^{n^2}$, the $n\times n$ doubly stochastic matrices form a polytope ($\omega_n$).
The  Birkhoff's theorem says that the polytope $\omega_n$ is the same as the polytope ($\delta_n$) generated by the permutation matrices.
A traditional proof of this result is by making use of a lemma  which ensures that
every doubly stochastic matrix has a positive diagonal (see, e.g.,  \cite[Lemma 8.7.1, p.~548]{HJ1.13}).
By a diagonal of an $n$-square matrix we mean a set of $n$ entries taken from
different rows and columns. The $n$-square doubly stochastic matrices having a positive diagonal form a polytope ($l_n$) too. Apparently, $\delta_n\subseteq l_n\subseteq \omega_n$.
Birkhoff's theorem asserts that the three polytopes $\omega_n$, $l_n$, and $\delta_n$ coincide.

In this paper, we consider the counterpart of the Birkhoff's theorem for higher dimensions.
A multidimensional array of numerical values is referred to as a {\em tensor} (see, e.g., \cite{KB09}).
It is also known as a hypermatrix \cite{Lim13}.
 Let $A=(a_{ijk})$ be an $n\times n\times n$ tensor (or an $n$-tensor cube). We call $A$ a {\em stochastic tensor} \cite{CLN14}, or
 {\em stochastic semi-magic cube} \cite{Mayathesis03}, or simply {\em stochastic cube} if
all $a_{ijk}\geq 0$, and
\begin{gather}
\sum_{i=1}^n a_{ijk}=1, \quad \forall j, k, \label{eq11}\\
\sum_{j=1}^n a_{ijk}=1, \quad \forall i, k, \label{eq12}\\
 \sum_{k=1}^n a_{ijk}=1, \quad \forall i, j.\label{eq13}
 \end{gather}

An $n$-tensor cube may be interpreted in terms of its slices \cite{KB09}.
By  a {\em slice} of a tensor $A$, we mean
a two-dimensional section of tensor $A$ obtained by fixing any one of the three indices. For
a $3$-tensor cube $A=(a_{ijk})$,  there are 9 slices,
each of which is  a square matrix.
An intersection of any two non-parallel slices is called a {\em line}
(also known as {\em fiber} or {\em tube}).
That is, a line is a one-dimensional section of a tensor; it is obtained by fixing all but one indices.
A {\em diagonal} of an $n\times n\times n$ tensor cube is a collection of
$n^2$ elements such that no two  lie on the same line. A nonnegative tensor is said to have
 a {\em positive diagonal} if all the elements of a diagonal are positive.
We  say
such a tensor has the {\em positive diagonal property}.

The Birkhoff theorem is about the matrices that are 2-way stochastic, while stochastic cubes are
3-way stochastic. Related works on partial or multiple stochasticity
  such as line or face stochasticity can be found in \cite{BrCs75, Csi70,
FiSw85}, 
including the recent ones
\cite{KB09} (on tensor computation),  \cite{Lim13} (a survey chapter on tensors and hypermatrices),
\cite{CLN14} (on extreme points of tensors), and \cite{CQZ13} (a survey on the spectral theory of nonnegative tensors).

Let $\Omega_{n}$ be the set of all $n\times n\times n$ stochastic tensors (i.e., semi-magic cubes).
It is evident that $\Omega_{n}$, regarded as a subset of $\Bbb R^{n^3}$, is  convex and compact
(since it is an intersection of finitely many closed
half-spaces and it is bounded).
 It is not difficult to show that
every permutation tensor (i.e.,  $(0, 1)$ stochastic tensor) is a vertex of $\Omega_{n}$.
Let $\Delta_n$ be the polytope generated (i.e., convex combinations of) by the $n\times n\times n$
permutation tensors and let $L_n$ be the set of all
$n\times n\times n$ stochastic tensors with the positive diagonal property. Obviously,
$$\Delta_n\subseteq L_n\subseteq \Omega_n.$$

If $n=2$, a straightforward computation yields  that $\Delta_n= L_n=\Omega_n.$
For $n\geq 3$, it is known that each of the above inclusion is proper (see Section 3).
Furthermore,  the number of vertices of $\Delta_n$ (i.e., permutation tensors) is equal to the number of Latin squares of order $n$
(\cite[Proposition 2.6]{CLN14}).

We assume $n\geq 2$ throughout the paper.  We show a close relation between
$L_n$ and $\Omega_n$, specially, the closure of $L_n$ is $\Omega_n$.
Moreover,  a  lower bound for the number of the vertices of $\Omega_{n}$ is available
in \cite[p.~34]{Mayathesis03}.
In this paper, we  give an upper bound for the number of vertices of $\Omega_{n}$.

\section{Vectorizing a cube}

For an $m\times n$ matrix $A$ with rows $r_1,  \dots, r_m$ and columns $c_1,  \dots, c_n$, let
$$vec_r(A)=\left (\begin{array}{cccc}
r_1^T\\
\vdots \\
r_m^T
\end{array} \right ), \quad vec_c(A)=\left (\begin{array}{cccc}
c_1\\
\vdots \\
c_n
\end{array} \right ).$$

Vectorizing, or ``vecing" for short, a matrix (respect to rows or columns) is a basic method in solving matrix equations. It also plays an important role in computation.
We present in this section how to ``vec" a cube. This may be useful in tensor computations which is a popular field in these days.

For $n\times n$ doubly stochastic matrices, we have the following fact by a direct verification.
Let $e_n=(1, \dots, 1)\in \Bbb R^n$.
 An $n\times n$ nonnegative matrix $S=(s_{ij})$ is doubly stochastic if and only if $S$ is a nonnegative matrix satisfying
$$(I_n\otimes e_n)vec_r(S)=e_n^T \quad \mbox{and} \quad (I_n\otimes e_n)vec_c(S)=e_n^T.$$

When a tensor cube is interpreted in terms of slices \cite[p.~458]{KB09},
we see that each slice of a stochastic tensor is a doubly stochastic square matrix.
An intersection of any two nonparallel slices is a {\em line} (fiber).
An $n\times n\times n$ tensor cube has $3n^2$ lines. Considering each line as a column vector of $n$ components, we stack all the lines
in the order of $i$, $j$, and $k$ directions (or modes), respectively,
to make a column vector of $3n^3$ components. We call such a vector the ``line vec" of the cube and denote it by
$vec_{\ell}(\cdot).$ Note that when vecing a 3rd order-$n$ dimensional tensor, every entry of the tensor is used 3 times.

For two cubes $A$ and $B$ of the same size and for any scalar $\alpha$, we have
$$vec_{\ell}(\alpha A+B)=\alpha vec_{\ell}(A) + vec_{\ell}(B)$$
and
$$\langle A, B\rangle =\frac13 \langle vec_{\ell}(A),  vec_{\ell}(B)\rangle,$$
where the left inner product is for tensors while the right one is for vectors.

What follows is a characterization of a stochastic tensor through ``vecing".

\begin{theorem} Let  $e_k $ be the all 1 row vector of $k$ components, where $k$ is a positive integer. An $n\times n\times n$ nonnegative cube $C=(c_{ijk})$ is stochastic if and only if
$$(I_{m}\otimes e_n) vec_{\ell} (C)=e_{m}^T, \quad \mbox{where $m=3n^2$}.$$
\end{theorem}

\begin{proof}
 By a direct verification.
 \end{proof}

\section{The convex set $L_n$}

It is known that if $A=(a_{ij})$ is an $n\times n$ doubly stochastic matrix, then $A$ has a positive diagonal; that is,
there exist $n$ positive entries of $A$
such that no two of these entries are on the same row and same column (the positive diagonal property).  Does the polytope of stochastic tensor cubes have the positive
diagonal property?
Let $A=(a_{ijk})$ be an $n\times n\times n$ stochastic cube.
Is it true that there always exist $n^2$ positive entries of $A$ such that no two of these entries lie on the same line? 
In short, does a stochastic tensor have the positive diagonal property?

It is easy to see that every permutation tensor is an extreme point of
$\Omega_n$. Apparently,  the set of nonnegative tensors of the same size forms a cone; that is,
if $A$ and $B$ (of the same size) have the positive diagonal property,  then so are $aA$ and $A+bB$ for any positive scalars $a, b$.
Obviously, every permutation tensor possesses the   positive diagonal property, so does any  convex combination of finitely many permutation tensors.
However,  some stochastic tensor cubes fail to have the positive diagonal property as the following example shows.

Note that $\Delta_{n}$ and $\Omega_{n}$ are convex and compact (in $\Bbb R^{n^3}$),
while $L_{n}$ is convex but not compact. In what follows, Example  \ref{ex1}  shows that
a stochastic tensor cube need not have a positive diagonal (unlike the case of doubly stochastic matrices); Example \ref{ex2}  shows
a stochastic tensor cube with the positive diagonal property need not be generated by permutation tensors.

\begin{example}\label{ex1}
Let $E$ be the $3\times 3 \times 3$ stochastic tensor cube:
\begin{figure}[h]
\begin{center}
\includegraphics[width=2in,totalheight=1.8in]{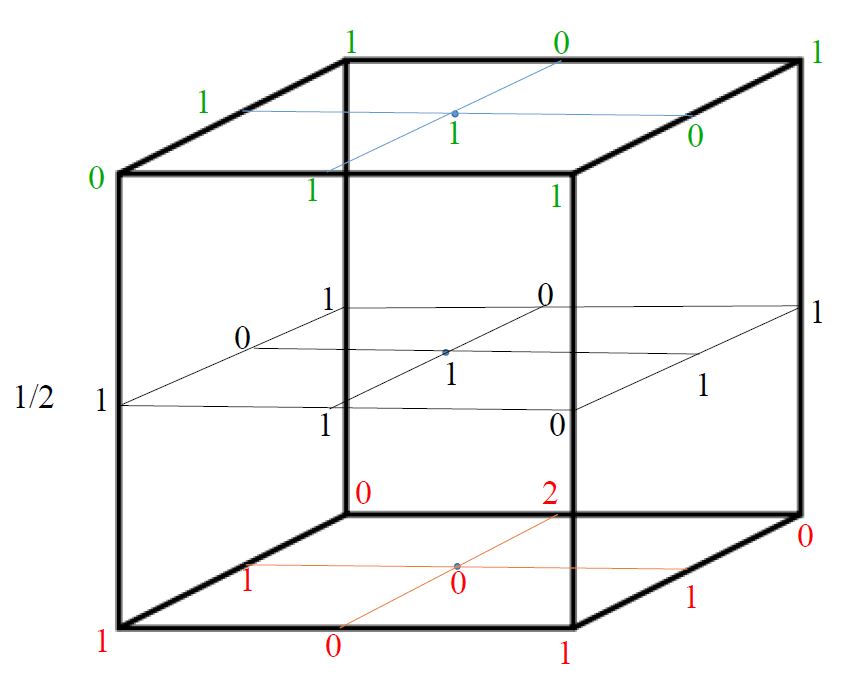}
\end{center}
\end{figure}
\linebreak
which can be ``flattened" to be a $3\times 9$ matrix
$$\frac12 \left  [
\begin{array}{ccccccccccc}
0 & 1 & 1 & \vdots  & 1 & 1 & 0 & \vdots &  1 & 0 & 1 \\
1 & 1 & 0 & \vdots  & 0 & 1 & 1 & \vdots &  1 & 0 & 1 \\
1 & 0&  1 & \vdots  & 1 & 0 & 1 & \vdots &  0 & 2 & 0
\end{array} \right ].$$

One may verify (by starting with the entry 2 at the position (3, 2, 3)) that $E$ has no positive diagonal and
$E$ is not a convex combination of the permutation tensors.
So $L_3\subset \Omega_3$. (In fact, $E$ is an extreme point of $\Omega_3$; see e.g.,
\cite{Mayathesis03}.)
\end{example}

\begin{example}\label{ex2}
Taking the stochastic tensor cube $F$ with the flattened matrix
$$\left  [
\begin{array}{ccccccccccc}
0 & 0.6 & 0.4 &     \vdots  & 1 & 0 & 0 &     \vdots &  0 & 0.4 & 0.6 \\
0.6& 0 & 0.4 & \vdots    & 0 & 0.4 & 0.6 & \vdots &  0.4 & 0.6 & 0 \\
0.4& 0.4&  0.2 & \vdots  & 0 & 0.6 & 0.4 & \vdots &  0.6 & 0 & 0.4
\end{array} \right ],$$
we see that $F\in L_3$ by choosing the  positive elements $\times $  as follows:
$$\left  [
\begin{array}{ccccccccccc}
 &   \times  &  &     \vdots  & \times  &  &   &     \vdots &    &  & \times \\
 \times &    &   & \vdots  &   &  & \times &          \vdots &  &  \times &   \\
  &  & \times & \vdots  &  & \times  &   & \vdots & \times   &  &
\end{array} \right ].$$

On the other hand, if $F$ is written as $x_1P_1+\cdots x_kP_k$, where
all $x_i$ are positive with sum 1, and each $P_i$ is a permutation tensor of the same size, then each
$P_i$ has 0 as its entry at position $(j_1, j_2, j_3)$ where $F_{j_1j_2j_3}=0$, that is, every $P_i$  takes the form
$$P_i= \left  [
\begin{array}{ccccccccccc}
0 & \ast & \ast &     \vdots  & 1& 0 & 0 &     \vdots &  0 & \ast & \ast \\
\ast& 0 & \ast & \vdots  & 0& {\star} & \ast & \vdots &  \ast & \ast & 0 \\
\ast & \ast&  \ast & \vdots  & 0& \ast & \ast & \vdots &  \ast& 0 & \ast
\end{array} \right ].$$
There exists only one such permutation tensor with 0 in the
 $(2, 2)$ position ($\star$) in the second slice.
Therefore,  $F\not\in \Delta_3$. So the inclusions $\Delta_3\subset L_3\subset \Omega_3$ are proper.
\end{example}

Next we show that
the closure of $L_{n}$ is $\Omega_{n}$ and also that
every interior point of $\Omega_n$ belongs to $L_n$. So $L_{n}$ is ``close" to $\Omega_{n}$
except some points of the boundary.

\begin{theorem} The closure of the set of all $n\times n\times n$ stochastic tensors having a positive diagonal
is the set of all $n\times n\times n$ stochastic tensors. In symbol,
$$\mbox{{\em cl}$(L_{n})$}=\Omega_{n}.$$
Moreover, every interior point (tensor) of $\Omega_n$ has a positive diagonal.
Consequently, a stochastic tensor that does not have the positive diagonal property belongs to the boundary $\partial \Omega_{n}$ of the polytope $\Omega_n$.
\end{theorem}

\begin{proof}
 Since $L_n\subseteq \Omega_n$, we have {\rm cl}$(L_{n})\subseteq \Omega_{n}.$
For the other way around, observe that every permutation tensor is in $L_{n}$. If $P, Q\in \Omega_n$,
where  $P$  is a permutation tensor, then $tP+(1-t)Q$ belongs to  $L_{n}$ for any $0<t\leq 1$.
Thus, for any $Q\in \Omega_n$, if we set $t=\frac{1}{m}$, we get
$\lim_{m\rightarrow \infty} \big (\frac1m P+(1-\frac1m )Q\big )=Q$.
This says that $\Omega_n\subseteq \mbox{{\rm cl}($L_{n})$}$.  It follows that
$\Omega_n=\mbox{{\rm cl}($L_{n})$}$.

We now show that every interior point of $\Omega_{n}$ lies in $L_n$.  Let $B$ be an interior point
of $\Omega_{n}$. Then there is an open ball, denoted by ${\mathcal{B}}(B)$, centered at $B$, inside $\Omega_{n}$. Take a permutation tensor
$A$, say, in $\Delta_{n}$. Then $tA+(1-t)B\in L_{n}$ for any $0<t\leq 1$. Suppose that
the intersection point of the sphere  $\partial {\rm cl}({\mathcal{B}}(B))$ with the line
$tA+(1-t)B$ is at $C$. Let $C'$ be the corresponding point of $C$ under the antipodal
mapping with respect to the center $B$. Then $B$ is between $A$ and $C'$, so B can be written as
$sA+(1-s)C'$ for some $0<s<1$.  By the above discussion,
$B=sA+(1-s)C'$ is in $L_n$. That is, every interior point of $\Omega_n$  lies in $L_n$.
\end{proof}

\section{An upper bound for the number of vertices}

The Birkhoff polytope, i.e., the set of doubly stochastic matrices, is the convex hull of its extreme points - the permutation matrices.
The Krein-Milman theorem (see, e.g., \cite[p. 96]{Zie95})
states that  every compact convex polytope is the convex hull of its vertices.
A fundamental question of polytope theory is that of an upper (or lower) bound for the number of vertices (or even faces). 
Determining the number of vertices (and faces) of a given polytope is a computationally difficult problem in general (see, e.g., the texts on polytopes \cite{Bro83} and \cite{Zie95}).

M. Ahmed et al \cite[p.~34]{Mayathesis03}
 gave a lower bound $(n!)^{2n}/n^{n^2}$ for the number of vertices (extreme points) of $\Omega_n$ through an algebraic combinatorial approach.
We present an upper bound and our approach is analytic.

\begin{theorem} Let $v(\Omega_n)$ be the number of vertices of the polytope $\Omega_n$.
Then
$$v(\Omega_n)\leq \frac{1}{n^3} \cdot {p(n)\choose n^3-1}, \quad \mbox{where $p(n)=n^3+6n^2-6n+2$.}$$
\end{theorem}

\proof  Considering $\Omega_n$  defined by (\ref{eq11})--(\ref{eq13}), we want to know the numbers of independent equations  (lines) that describe $\Omega_n$. For each horizontal slice (an $n\times n$ doubly stochastic matrix),  $2n-1$ independent lines are needed and sufficient. So there are $n(2n-1)$ independent horizontal lines from $n$ horizontal slices. Now consider the vertical lines, there are $n^2$ of vertical lines. However,  $(2n-1)$ of them, say,  on the most right and back, have been determined by the horizontal lines (as each line sum is 1). Thus, $n^2-(2n-1)=(n-1)^2$ independent vertical lines are needed. So there are $n(2n-1)+(n-1)^2=3n^2-3n+1$ independent lines in total to define the tensor cube.
It follows that  we  can view $\Omega_n$  defined by (\ref{eq11})-(\ref{eq13})  as the set of all vectors
$x=(x_{ijk})\in \Bbb R^{n^3}$  satisfying 
\begin{gather}
\sum_{i=1}^nx_{ijk}=1, \quad 1\leq j\leq n,\; 1\leq k\leq n, \label{eq1} \\
\sum_{j=1}^nx_{ijk}=1, \quad 1\leq i\leq n,\; 1\leq k\leq n-1, \label{eq2} \\
\sum_{k=1}^nx_{ijk}=1, \quad 1\leq i\leq n-1,\; 1\leq j\leq n-1, \label{eq3} \\
\ x_{ijk}\geq 0, \quad 1\leq i, j, k\leq n. \label{eq4}
\end{gather}\smallskip

We may rewrite (\ref{eq1})--(\ref{eq3}) and (\ref{eq4}) respectively as
$$Ax=u, \quad Bx\geq 0,$$
where $A$ is a $(3n^2-3n+1)\times n^3$ $(0, 1)$ matrix, $u$ is the all 1 column vector in $\Bbb R^{3n^2-3n+1}$,  and $B$ is an $n^3\times n^3$ $(0, 1)$ matrix.
Let $m=n^3$.

A subset  of $\Bbb R^m$ is a convex hull of a finite set if and only if it is a bounded intersection of closed half-spaces \cite[p. 29]{Zie95}.
The polytope  $\Omega_n$ is generated by the $p=n^3+6n^2-6n+2$ half-spaces defined by the
linear inequalities $Ax\geq u$, $Ax\leq u$, and $Bx\geq 0$, $x\in \Bbb R^m$. Let $e$ be a vertex of $\Omega_n$.
We claim that at least $m$ equalities $he=1$ or 0 hold, where $h$ is a row of $A$ or row of $B$, that is,
$Ce=w$, where $C$ is a $k\times n^3$ ($k\geq m$)  matrix consisting some rows of $A$ and some rows of $B$, and
$w$ is a $(0, 1)$ column vector.
 If, otherwise, $k$ equalities hold for $k<m$,
let $K=\{x\in \Bbb R^m \mid Cx=w\}$. $K$ is an affine space and $e\in K$. Since $C$ is a $k\times m$ matrix,
the affine space $K$ has dimension at least $m-k \geq 1$. Let
$O=\{x\in \Bbb R^m \mid B'x>0\} \cap K$, where $B'$ is a submatrix of $B$ for which $B'e>0$.
$O$ is open in $K$. Since $e\in O$, $e$ is an interior point of $O$, thus it cannot be an extreme
point of $\Omega_n$. We arrive at that every extreme point $e$ lies on at least $m$ supporting hyperplanes $h(x):=hx=w$ in (\ref{eq1})--(\ref{eq4}) that define $\Omega_n$.

To show the upper bound, we use induction on $n$ by reducing the problem to
a polytope (a supporting hyperplane of $\Omega_n$) of lower dimensions.
Let $V_m$ be the maximum value of the vertices of polytopes formed by any $p$ supporting hyperplanes in $\Bbb R^m$. We show
$$V_m\leq   \frac{1}{m} {p \choose m-1} = \frac{1}{n^3} {n^3+6n^2-6n+2\choose n^3-1}$$

For $m=8$, i.e. $n=2$, it is easy to check as $\Omega_2$ has only two vertices. Assume that
 the upper bound inequality holds for the polytopes in the spaces $\Bbb R^k$, $k<m$.
 $\Omega_n$ is formed by   $p$  supporting hyperplanes
  $H_t=\{x \mid h_t(x)=u\}$, $t=1, \dots, p$.
Since $H_t$ is a face of $\Omega_n$,
the vertices of $\Omega_n$ lying in $H_t$ are the vertices of $H_t$. As $H_t$ has smaller dimension than $\Omega_n$ \cite[p. 32]{Bro83} and it is formed
by at most $p-1$ hyperplanes, by the induction hypothesis,
$H_t$ has at most $V_{m-1}$ vertices, each of which lies in at least $m$ hyperplanes. We arrive at
\begin{eqnarray*}
{}\hspace{1.4in} 
v(\Omega_n) & \leq & \frac{1}{m} \sum_{t=1}^p v (H_t)\\
& \leq & \frac{p}{m}\cdot V_{m-1}\\
& =& \frac{p}{m}\cdot  \frac{1}{m-1}\cdot  {{p-1} \choose {m-2}}\\
& = & \frac{1}{m} {p \choose m-1}\\
& = & \frac{1}{n^3} {n^3+6n^2-6n+2\choose n^3-1}.
{}\hspace{1.4in}\qed 
\end{eqnarray*}
\\
{\bf Acknowledgement.}
 The work of Haixia Chang was done during the academic year 2014-2015 when she was a Visiting Professor at Nova Southeastern University;
her work was partially supported by National Natural Science Foundation of China (No. 11501363). Fuzhen Zhang is thankful to
Richard Stanley, Zejun Huang, and Rajesh Pereira for drawing his attention to a number of references.
We would also like to express our thanks to the referee and Zhongshan Li for some corrections and discussions.
The project was partially supported by National Natural Science Foundation of China (No.~11571220).



\end{document}